\documentclass[]{article}

\usepackage[english]{babel}
\usepackage{amssymb,amsmath, amsthm, amscd}

\usepackage{bbm}
\usepackage{amssymb}

\newcommand{\N}{\mathbb{N}}

\newcommand{\Ff}[1]{\mathbb{F}_{#1}}

\newcommand{\polr}[3]{\field{#1}[#2_1,\ldots,#2_{#3}]}
\newcommand{\field}[1]{\mathbbm{#1}}
\newcommand{\hide}[1]{}

\newcommand{\Bb}[2]{\beta_{#1,#2}}
\newcommand{\Bnum}[3]{\beta_{#1,#2}(I_{#3})}
\newcommand{\nnn}{\mathbb{N}_{0}^{n}}
\newcommand{\nn}{\mathbb{N}_{0}}
\newcommand{\n}[1]{\mathbb{N}}

\newcommand{\feit}[1]{\mathbf{#1}}
\newcommand{\elo}[2]{#1_{(#2)}}
\newcommand{\dual}[1]{\overline{#1}}

\theoremstyle{plain}
\newtheorem{theorem}{Theorem}[section]
\newtheorem{lemma}{Lemma}[section]
\newtheorem{proposition}{Proposition}[section]
\newtheorem{corollary}{Corollary}[section]

\theoremstyle{definition}
\newtheorem{definition}{Definition}[section]
\newtheorem{example}{Example}[section]

\theoremstyle{remark}

\DeclareMathOperator{\col}{col} 
\DeclareMathOperator{\Supp}{Supp} 
\DeclareMathOperator{\im}{im} 
\DeclareMathOperator{\sign}{sign}

\usepackage[affil-it]{authblk}

\begin{document}
\title{A generalization of weight polynomials to matroids\footnote{The original publication is available at http://www.sciencedirect.com/science/article/pii/S0012365X15003568}}
\author{Trygve Johnsen, Jan Roksvold\thanks{Corresponding author. E-mail address: \texttt{jan.n.roksvold@uit.no}}, Hugues Verdure}
\affil{Department of Mathematics and Statistics, UiT The Arctic University of Norway,\\ N-9037 Troms\o, Norway}

\maketitle
\begin{abstract}
\noindent Generalizing polynomials previously studied in the context of linear codes, we define weight polynomials and an enumerator for a matroid $M$. Our main result is that these polynomials are determined by Betti numbers associated with $\nn$-graded minimal free resolutions of the Stanley-Reisner ideals of $M$ and so-called elongations of $M$. Generalizing Greene's theorem from coding theory, we show that the enumerator of a matroid is equivalent to its Tutte polynomial.
\end{abstract}

\hide{}

\section{Introduction}
For a linear $[n,k]$-code $C$ over $\Ff{q}$, let $A_{C,j}$ denote the number of words of weight $j$ in $C$. The Hamming weight enumerator \[W_{C}(X,Y)=\sum_{j=0}^{n}A_{C,j}X^{n-j}Y^{j}\] has important applications in the theory of error-correcting codes, where it amongst other things determines the probability of having an undetected error (see \cite[Proposition 1.12]{JP}). 

Let $M(H)$ denote the vector matroid associated to a parity-check matrix $H$ of $C$. The connection \begin{equation}\label{Greene's}W_{C}(X,Y)=(X-Y)^{n-k}Y^{k}t_{M(H)}\left(\frac{X}{Y},\frac{X+(q-1)Y}{X-Y}\right)\end{equation} between the Hamming weight enumerator of an $\Ff{q}$-code and the specialization of its associated Tutte polynomial on the hyperbola $(x-1)(y-1)=q$ was first established in Greene's paper \cite{Gre}, and we shall therefore refer to equation (\ref{Greene's}) as \emph{Greene's theorem}. 

For $Q$ a power of $q$, the set of all $\Ff{Q}$-linear combinations of words of $C$ is itself a linear code. This code is commonly referred to as the extension of $C$ to $\Ff{Q}$, and is denoted by $C\otimes_{\Ff{q}}\Ff{Q}$. 
In \cite{JP}, it is found that the number $A_{C,j}(Q)$ of words of weight $j$ in $C\otimes_{\Ff{q}}\Ff{Q}$ can be expressed in terms of the initial code $C$, as a polynomial in $Q$. This leads the authors to the definition of an \emph{extended} weight enumerator $W_{C}(X,Y,Q)$ for $C$, with the desired property that \[W_{C}(X,Y,Q)=W_{C\otimes_{\Ff{q}}\Ff{Q}}(X,Y).\] The polynomial $W_{C}(X,Y,Q)$ is then, in turn, shown to be equivalent to the Tutte polynomial of $M(H)$ -- thereby extending Greene's theorem.  

Our primary objective in this article is to demonstrate that the polynomial $A_{C,j}(Q)$ is determined by the Betti numbers associated to $\nn$-graded minimal free resolutions of the Stanley-Reisner ideals of $M(H)$ and its so-called \emph{elongations}. This is intended to serve as one brick in the bridge being built between commutative combinatorial algebra and the theory of linear codes. The result can also be seen as a continuation of the work done in \cite{JV}, where it is demonstrated that the Betti numbers associated to an $\nn$-graded minimal free resolution of the Stanley-Reisner ideal of $M(H)$ determine the higher Hamming weight hierarchy of $C$. 

It seemed natural to begin the pursuit of the above-stated objective by generalizing the polynomial $A_{C,j}(Q)$ to a polynomial $P_{M,j}(Z)$ defined for all matroids, not only those stemming from a linear code, but of course with the property that $A_{C,j}(Q)=P_{M(H),j}(Q)$. Having defined such a polynomial $P_{M,j}(Z)$, it is equally natural to define and investigate a more general \emph{matroidal} enumerator \[W_{M}(X,Y,Z)=\sum_{j=0}^{n}P_{M,j}(Z)X^{n-j}Y^{j}.\] 

Our second objective is to extend Greene's theorem from codes to matroids by way of this matroidal weight enumerator. Since its discovery, Greene's theorem has been generalized, specialized, and extended in several ways. For example, in addition to the already mentioned equivalence between the Tutte polynomial and the extended weight enumerator of a linear code, it was demonstrated in \cite[Theorems 4 and 5]{Bri} and (independently) in \cite[Theorem 3.3.5]{Jur} that the Tutte polynomial and the set of so-called higher weight enumerators of a linear code determine each other as well. Related results and methods can also be found in ~\cite{Bar}, where the connection between the weight enumerator and the Tutte polynomial is used to establish bounds on all-terminal reliability of vectorial matroids. In addition, \cite{Bar} provides new proofs of Greene's theorem itself, and shows how the weight polynomial and the partition polynomial of the Potts model are related. The connection between the weight enumerator 
and the Tutte polynomial is also used in \cite[Corollaries 10, 11 and 12]{Rei} when looking at two-variable coloring formulas for graphs. A generalization of Greene's theorem is given in \cite[Theorem 9.4]{Ver} to latroids, which are useful for studying codes over rings. 

As can be seen in \cite[p.~131]{Du}, the Tutte polynomial of a matroid determines its higher weights. Thus we already know that the polynomials $P_{M,j}$, being equivalent to the Tutte polynomial, must determine the higher weights of $M$ as well -- at least indirectly. We shall see that they do so in a very simple and accessible way. 

\subsection{Structure of this paper}
\begin{itemize} 
\item Section \ref{Secprelim} contains definitions and results used later on. 
\item In Section \ref{Pol} we look at the number of codewords in the extension of a linear code $C$ over $\Ff{q}$ -- as a polynomial in $q^r$.
\item In Section \ref{SecGWP}, we generalize the polynomial from Section \ref{Pol} to matroids, and use these generalized weight polynomials to define a matroidal enumerator. We proceed to demonstrate that this enumerator is \emph{equivalent} to the Tutte polynomial of $M$.
\item In Section \ref{BettiSect} we prove our main result: The generalized weight polynomials are determined by the Betti numbers associated to $\nn$-graded minimal free resolutions of the Stanley-Reisner ideal of $M$ and the elongations of $M$.
\item In Section \ref{Counter} we shall see a counterexample showing that the converse of our main result is not true; the generalized weight enumerators do not determine the $\nn$-graded Betti numbers of the Stanley-Reisner ideal of $M$.
\end{itemize} 
 
\section{Preliminaries}\label{Secprelim}
\subsection{Linear codes and weight enumerators}
A linear $[n,k]$-code $C$ over $\Ff{q}$ is, by definition, a $k$-dimensional subspace of $\Ff{q}^n$. The elements of this subspace are commonly referred to as \emph{words}, and any $k\times n$ matrix whose rows form a basis for $C$ is referred to as a \emph{generator matrix}. Thus a linear code typically has several generator matrices.

The \emph{dual code} is the orthogonal complement of $C$, and is denoted by $C^{\perp}.$ A \emph{parity-check matrix} of $C$ is a $(n-k)\times n$-matrix with the property that \[Hx^{T}=0\Leftrightarrow x\in C.\] It is easy to see that $H$ is a parity check matrix for $C$ if and only if $H$ is a generator matrix for $C^{\perp}$.

\subsection{Puncturing and shortening a linear code}
Let $C$ be a linear code of length $n$, and let $J\subseteq\{1\ldots n\}$.
\begin{definition}
The \emph{puncturing of $C$ in $J$} is the linear code obtained by eliminating the coordinates indexed by $J$ from the words of $C$. \end{definition}
\begin{definition}
\[C(J)=\{w\in C: w_j=0 \text{ for all }j\in J\}.\] 
\end{definition}
Clearly, $C(J)$ is itself a linear code.
\begin{definition}
The \emph{shortening of $C$ in $J$} is the puncturing of $C(J)$ in $J$.
\end{definition}

\subsection{Matroids}There are numerous equivalent ways of defining a matroid. We choose to give here the definition in terms of independent sets. For an introduction to matroid theory in general, we recommend e.g.~\cite{Oxl}.
\begin{definition}A \emph{matroid} $M$ consists of a finite set $E$ and a set $I(M)$ of subsets of $E$ such that:
\begin{itemize} \item $\emptyset \in I(M).$ \item If $I_{1}\in I(M)$ and $I_{2}\subseteq I_{1}$, then $I_{2}\in I(M)$. \item If $I_{1},I_{2}\in I(M)$ and $|I_{1}|>|I_{2}|$, then there is a $x\in I_{1}\smallsetminus I_{2}$ such that $I_{2}\cup x\in I(M)$.\end{itemize} 
\end{definition}
The elements of $I(M)$ are referred to as the \emph{independent sets} (of $M$). The \emph{bases} of $M$ are the independent sets that are not contained in any other independent set. In other words, the \emph{maximal independent sets}. Conversely, given the bases of a matroid, we find the independent sets to be those sets that are contained in a basis. We denote the bases of $M$ by $B(M)$. It is a fundamental result that all bases of a matroid have the same cardinality.

The dual matroid $\dual{M}$ is the matroid on $E$ whose bases are the complements of the bases of $M$. Thus \[B(\dual{M})=\{E\smallsetminus B:B\in B(M)\}.\] 
\begin{definition}For $\sigma\subseteq E$, the rank function $r_{M}$ and nullity function $n_{M}$ are defined by \[r_{M}(\sigma)=\max\{|I|:I \in I(M), I\subseteq \sigma\},\] and \[n_{M}(\sigma)=|\sigma|-r_{M}(\sigma).\]\end{definition}

Whenever the matroid $M$ is clear from the context, we omit the subscript and write simply $r$ and $n$. Note that a subset $\sigma$ of $E$ is independent if and only if $n(\sigma)=0$. The rank $r(M)$ of $M$ itself is defined as $r(M)=r_{M}(E)$. 

We let $\dual{r}$ and $\dual{n}$, respectively, denote the rank- and nullity function of $\dual{M}$, and point out that \begin{equation}\label{dual}\dual{r}(\sigma)=|\sigma|+r(E\smallsetminus \sigma) - r(E).\end{equation}

\begin{definition}
If $\sigma\subseteq E$, then $\{I\subseteq\sigma:I \in I(M)\}$ form the set of independent sets of a matroid $M_{|\sigma}$ on $\sigma$. We refer to $M_{|\sigma}$ as the \emph{restriction of $M$ to $\sigma$}.  
\end{definition}

\begin{definition}\label{di}The \emph{higher weights} $\{d_i\}$ of $M$ are defined by \[d_i=\min\{|\sigma|:\sigma\subseteq E(M) \text{ and } n(\sigma)=i\}.\]
\end{definition}

\begin{definition}\label{Tuttedef}
The \emph{Tutte polynomial} of $M$ is defined by \[t_{M}(X,Y)=\sum_{\sigma\subseteq E}(X-1)^{r(E)-r(\sigma)}(Y-1)^{|\sigma|-r(\sigma)}.\] 
\end{definition} 
It carries information on several invariants of $M$. For example $t_{M}(1,1) $ counts the number of bases of $M$, while $t_{M}(2,1)$ is the number of independent sets.

\begin{definition}Let $f_i$ denote the number of independent sets of cardinality $i$. 
The \emph{reduced Euler characteristic} $\chi(M)$ of $M$ is defined by \[\chi(M)=-1+f_1-f_2+\cdots+(-1)^{r(M)-1}f_{r(M)}.\] 
\end{definition}
It is straightforward to verify that $\chi(M)=(-1)^{r(M)-1}t_{M}(0,1).$

Without any loss of generality we shall throughout this article assume that \[E=\{1,\ldots,n\}.\] Furthermore, we shall frequently identify $\sigma\subseteq E$ with its indicator vector in $\{0,1\}^{n}$ whose $i^{th}$ coordinate is $1$ if and only if $i\in \sigma$. The expression $|\sigma|$ should, however, always be interpreted as the number of elements in $\sigma$, or, equivalently, as the number of elements in the support of the corresponding indicator vector. 
\begin{example}[$U(r,n)$]
The set of all cardinality-$r$ subsets of $E$ form the set of bases for a matroid $U(r,n)$ on $E$. We refer to $U(r,n)$ as the \emph{uniform matroid} of rank $r$ on an $n$-element set. Observe that $I\subseteq E$ is independent in $U(r,n)$ if and only if $|I|\leq r$. 

Clearly, we have $d_i(U(r,n))=r+i,$ for $1\leq i \leq n-r.$ And it is equally clear that \[\chi(U(r,n))=\sum_{i=0}^{r}(-1)^{i+1}\binom{n}{i}.\] 

As for the Tutte polynomial, note that for $\sigma\subseteq E$ with $|\sigma|<r$ we have $|\sigma|-r(\sigma)=0.$ While for those $\sigma$ with $|\sigma|>r$ we have $r(E)-r(\sigma)=0.$ For the $\binom{n}{r}$ subsets $\sigma$ with $|\sigma|=r$, both $|\sigma|-r(\sigma)$ and $r(E)-r(\sigma)$ are equal to $0$. Thus \[t_{U(r,n)}(X,Y)=\sum_{i=0}^{r-1}\binom{n}{i}(X-1)^{r-i}+\binom{n}{r}+\sum_{i=r+1}^{n}\binom{n}{i}(Y-1)^{i-r}.\]
\end{example}

\subsection{From linear code to matroid}
Let $A$ be an $m\times n$ matrix over some field $\field{k}$. Let $E$ be the set of column labels of $A$. It is easy to verify that if we take as independent sets those subsets of $E$ that correspond to a set of $\field{k}$-linearly independent columns, this constitutes a matroid on $E$. We refer to this as the \emph{vector matroid} of $A$ and denote it $M(A).$ 

Note that if $G$ and $G'$ are two generator matrices for the linear code $C$, then $M(G)=M(G')$. The same goes for parity-check matrices, of course. It therefore makes sense to speak of \emph{the} matroid corresponding to a generator (or parity-check) matrix of $C$, and to write $M(G)$ and $M(H)$ without specifying $G$ or $H$. Thus to a linear code $C$, with generator matrix $G$ and parity-check matrix $H$, there naturally correspond two matroids: $M(G)$ and $M(H)$. We shall mostly consider $M(H)$, but this is not very crucial since duality results abound and $M(H)=\dual{M(G)}$. 

Note that $r(M(G))=\dim(C)$, while $r(M(H))=\dim(C^{\perp})$, and that $d_{1}(M(H))$ is equal to the minimum distance of $C$.

\begin{example}\label{RunEx}
Let $C$ be the $[7,4]$-code over $\Ff{5}$ with parity-check matrix \[H= \left( \begin{array}{ccccccc}
1 & 0 & 0 & 3 & 3 & 3 & 4 \\
0 & 1 & 0 & 0 & 2 & 2 & 0\\
0 & 0 & 1 & 4 & 4 & 4 & 4\end{array} \right).\] Then $M(H)$ will be a matroid on $E=\{1,\ldots,7\}$. The columns $\left[\begin{array}{c}1\\0\\0\end{array}\right],\left[\begin{array}{c}0\\0\\1\end{array}\right]$ and $\left[\begin{array}{c}3\\2\\4\end{array}\right]$ form a \emph{maximal} linearly independent set of columns, so $\{1,3,6\}$ is a basis for $M(H)$, and $r(M(H))=3$. The full set of bases is \begin{flalign*}B(M(H))=\big\{&\{ 1, 3, 6 \},\{ 1, 3, 5 \},\{ 1, 2, 6 \},\{ 2, 3, 6 \},\{ 1, 2, 5 \},\{ 1, 5, 7 \},\{ 3, 6, 7 \},\{ 2, 4, 7 \},\\ &\{ 1, 4, 6 \},\{ 2, 3, 4 \},\{ 4, 6, 7 \},\{ 1, 2, 3 \}, \{ 1, 2, 7 \},\{ 3, 4, 5 \},\{ 1, 6, 7 \},\{ 1, 4, 5 \},\\ &\{ 1, 2, 4 \},\{ 2, 3, 7 \},\{ 4, 5, 7 \},\{ 3, 5, 7 \},\{ 2, 6, 7 \},\{ 2, 5, 7 \},\{ 2, 3, 5 \},\{ 3, 4, 6 \}\big\}.\end{flalign*} 
\end{example}

\subsection{The elongation of $M$ to rank $r(M)+i$}\label{Mi}
Let $M$ be a matroid on $E=\{1,\ldots,n\}$. 
\begin{definition}
For $0\leq i\leq n-r(M)$, let $\elo{M}{i}$ be the matroid whose independent sets are given by $I(\elo{M}{i})=\{\sigma\in E: n(\sigma)\leq i\}$. 
\end{definition}
It is not difficult to verify that $\elo{M}{i}$ is indeed a matroid \cite[Section 1.3,~ex.6]{Oxl}. Note that $\elo{M}{0}=M$, and that $B(\elo{M}{n-r(M)})=\{E\}$.

The following is straightforward:\begin{proposition}\label{ni1}
For $\sigma\subseteq E,$ we have \begin{equation}\label{ri}r_{\elo{M}{i}}(\sigma)=\begin{cases}
r(\sigma)+i, & \text{$n(\sigma)>i$.} \\
|\sigma|, & \text{$n(\sigma)\leq i$.} \\
\end{cases}\end{equation}
And \begin{equation}\label{ni}n_{\elo{M}{i}}(\sigma)=\begin{cases}
n(\sigma)-i, & \text{$n(\sigma)>i$.} \\
0, & \text{$n(\sigma)\leq i$.} \\
\end{cases}\end{equation}
\end{proposition}
\hide{\begin{proof}
The equations (\ref{ri}) and (\ref{ni}) are clearly equivalent. We prove (\ref{ri}). 

Let $\sigma\subseteq E$. If $n(\sigma)\leq i$, then $r_i(\sigma)=|\sigma|;$ we therefore assume that $n(\sigma)>i$, and let $I\subseteq \sigma$ with $|I|=r(\sigma).$ Let $J\subseteq \sigma\smallsetminus I$ with $|J|=n(\sigma)-i.$ Clearly, \[|\sigma\smallsetminus J|=\max\{|X|:X\subseteq \sigma, n(X)=i\},\] and \[r(\sigma\smallsetminus J)=|I|=r(\sigma).\] Thus \[r_i(\sigma)=|\sigma\smallsetminus J|=r(\sigma\smallsetminus J)+n(\sigma\smallsetminus J)=r(\sigma)+i.\qedhere\] 
\end{proof}}

By definition we have $r_{i}(\elo{M}{i})=r_{i}(E)$. It thus follows from Proposition \ref{ni1} that 
\begin{equation}\label{rmi}
r_i(\elo{M}{i})=r(M)+i. 
\end{equation}
The matroid $\elo{M}{i}$ is commonly referred to as the \emph{elongation} of $M$ to rank $r(M)+i$.  

If $\sigma\subseteq E$ then the rank function of $M_{|\sigma}$ is the restriction of $r_{M}$ to subsets of $\sigma$. We point out, for later use, that this implies \begin{equation}\label{switch}
(\elo{M}{i})_{|\sigma}=(M_{|\sigma})_{(i)}.
\end{equation}

\subsection{The Stanley-Reisner ideal, Betti numbers, and the reduced chain complex}
Let $M$ be a matroid on $E=\{1,\ldots,n\}$. Let $\field{k}$ be a field.

\begin{definition}
A \emph{circuit} of $M$ is a subset $C$ of $E$ with the property that $C$ is not itself independent, but $C\smallsetminus x$ is independent for every $x\in C$.  
\end{definition}
In other words, the circuits of a matroid are the minimal dependent sets, while the independent sets are precisely those that do not contain a circuit.

Let $S=\polr{k}{x}{n}$.
\begin{definition}[Stanley-Reisner ideal]
Let $I_{M}$ be the ideal in $S$ generated by monomials corresponding to circuits of $M$. That is, let \[I_{M}=\langle x_{j_{1}}x_{j_2}\cdots x_{j_s}:\{j_{1}j_{2},\ldots,j_{s}\} \text{ is a circuit of } M \rangle.\] We refer to $I_{M}$ as the \emph{Stanley-Reisner ideal} of $M$.  
\end{definition}

It is clear that, viewed as an $S$-module, the Stanley-Reisner ideal permits both the standard $\nn$-grading and the standard $\nnn$-grading \cite[Section 6.3]{CLO}. 

\begin{definition}
For $\feit{a}\in \nnn$, let $S(-\feit{a})$ denote the $S$-module obtained by shifting the gradation of $S$, seen as an $\nnn$-graded module, by $\feit{a}$. \end{definition}

\begin{definition}For $j\in\nn$, let $S(-j)$ denote the $S$-module obtained by shifting the gradation of $S$, seen as an $\nn$-graded module, by $j$.  
\end{definition}
Note that $S(-\feit{a})$ is isomorphic to $S\feit{x}^{\feit{a}}$ as an $\nnn$-graded $S$-module, while $S(-j)$ is isomorphic to $\langle x^{j}\rangle_{S}$ as an $\nn$-graded $S$-module.
\begin{definition}An $\nnn$-graded $S$-module $F$ is said to be \emph{free} if it is of the form \[F=S(-\mathbf{a_{1}})\oplus S(-\mathbf{a_{2}})\cdots\oplus S(-\mathbf{a_{r}}),\] for some $\feit{a_1},\feit{a_2},\ldots,\feit{a_r}\in \nnn$.\end{definition}
And likewise:
\begin{definition} An $\nn$-graded $S$-module $F$ is said to be \emph{free} if it is of the form \[F=S(-j_{1})\oplus S(-j_{2})\cdots\oplus S(-j_{r}),\] for some $j_1,j_2,\ldots,j_r\in \nn$.\end{definition} 

\begin{definition}A chain of $S$-modules and $S$-homomorphisms \[\begin{CD}
\cdots@<<<X_{i-1}@<\phi_{i}<<X_{i}@<<<\cdots\end{CD}\] is said to be a \emph{complex} if $\im\phi_{i}\subseteq \ker\phi_{i-1}$ for each $i$. Furthermore, the complex is said to be \emph{minimal} whenever $\im\phi_i\subseteq \langle x_1,x_2,\ldots,x_n\rangle X_{i-1}$.\end{definition}
A complex is said to be \emph{exact at homological degree $i$} if $\im\phi_{i}=\ker\phi_{i-1}$.
Bringing these concepts together, we have:
\begin{definition}\label{minfree}An \emph{$\nnn$-graded minimal free resolution} of an $\nnn$-graded $S$-module $N$ is a minimal left complex
\begin{equation}\label{minres}\begin{CD}
0@<<<F_{0}@<\phi_1<<F_{1}@<\phi_2<<F_{2}@<<<\cdots@<\phi_l<<F_{l}@<<<0 
\end{CD}\end{equation}
where \[F_i=\bigoplus_{\mathbf{a}\in\nnn}S(-\mathbf{a})^{\beta_{i,\mathbf{a}}},\] which is exact everywhere except for in $F_0$, where $F_0/\im\phi_1\cong N.$ We also require the homomorphisms $\phi_{i}$ to be degree-preserving, in that degree $\feit{a}$ elements of $F_{i}$ are sent to degree $\feit{a}$ elements of $F_{i-1}$.\end{definition} 

It is straightforward to verify that the resolution being \emph{minimal} implies $\feit{a}\in\{0,1\}^{n}$ for each $\feit{a}$ appearing in (\ref{minres}). 

Replacing ``$\nnn$-graded'' with ``$\nn$-graded'' and setting \[F_i=\bigoplus_{j\in\nn}S(-j)^{\beta_{i,j}}\] in Definition \ref{minfree} gives us the definition of an \emph{$\nn$-graded minimal free resolution} of $N$. 

The $\Bb{i}{\feit{a}}$ and $\Bb{i}{j}$ are referred to as the $\nnn$-graded and $\nn$-graded \emph{Betti numbers} of $N$, respectively. Sometimes we want to emphasize the module $N$, and write $\Bb{i}{\feit{a}}(N)$ or $\Bb{i}{j}(N)$. Hilbert's Syzygy Theorem states that the \emph{length} $l$ of (\ref{minres}) is less than or equal to $n$. We shall here only be looking at minimal free resolutions of the Stanley-Reisner ideal $I_{M}$. Since $M$ is a matroid, these all have length $n-r(M)-1$ (see e.g.~\cite[Corollary 3(b)]{JV}). 

It follows from \cite[Theroem A.2.2]{HH} that \emph{the Betti numbers associated with a ($\nn$- or $\nnn$-graded) minimal free resolution are unique}, in that any other minimal free resolution must have the same Betti numbers. We may therefore, without ambiguity, write \[\Bb{i}{j}=\sum_{|\Supp(\feit{a})|=j}\Bb{i}{\feit{a}}.\] Note that for an empty ideal all the (graded or ungraded) Betti numbers are zero. This is for example always the case with $I_{\elo{M}{n-r(M)}}$ since $\elo{M}{n-r(M)}$ has no circuits. 

\begin{definition}
Let $I_{i}(M)$ denote the set consisting of those independent sets of $M$ that have cardinality $i$, and let $\field{k}^{I_{i}(M)}$ be the free $\field{k}$-vector space on $I_{i}(M)$. The \emph{(reduced) chain complex} of $M$ over $\field{k}$ is the complex \[\minCDarrowwidth18pt\begin{CD}0
@<<<\field{k}^{I_{0}(M)}@<\delta_{1}<<\cdots@<<<\field{k}^{I_{i-1}(M)}@<\delta_{i}<<\field{k}^{I_{i}(M)}@<<<\cdots@<\delta_{r(M)}<<\field{k}^{I_{r(M)}(M)}@<<<0\end{CD},\] where the boundary maps $\delta_{i}$ are defined on independent sets of $M$ of size $i$ as follows: With the natural ordering on $E$, set $\sign(j,\sigma)=(-1)^{r-1}$ if $j$ is the $r^{th}$ element of $\sigma\subseteq E$, and let \[\delta_{i}(\sigma)=\sum_{j\in \sigma}\sign(j,\sigma)\;\sigma\smallsetminus j.\] Extending $\delta_{i}$ $\field{k}$-linearly, we obtain a $\field{k}$-linear map from $I_{i}(M)$ to $I_{i-1}(M)$.
\end{definition}

\begin{definition}
The $i^{th}$ \emph{reduced homology} of $M$ over $\field{k}$ is the vector space \[H_{i}(M;\field{k})=\ker(\delta_{i})/\im(\delta_{i+1}).\]
\end{definition}

In proving our main result (Theorem \ref{main}), we shall draw upon the following two results, the first of which is a concatenation of \cite[Proposition 7.4.7 (i) and Proposition 7.8.1]{Bjoe}.
\begin{theorem}\label{BjoLemma} Let $H_{i}(M;\field{k})$ denote the $i$-th homology of $M$ over $\field{k}$. Then 
\[H_{i}(M;\field{k})=\begin{cases}
		\field{k}^{(-1)^{i-1}\chi(M)}, & i=r(M)\\
		0, & i\neq r(M). \\
	\end{cases}\]
\end{theorem}

\begin{theorem}[Hochster's formula]\label{Hochster}\[\Bnum{i}{\sigma}{M}=\dim_{\field{k}}H_{|\sigma|-i-1}(M_{|\sigma};\field{k}).\]
\end{theorem}
\hide{\begin{proof}
See \cite[Corollary 5.12]{MS}. 
\end{proof}}

We would like to point out, for later use, that Theorems \ref{BjoLemma} and \ref{Hochster} combined imply  
\begin{equation}\label{eple}\sum_{i=0}^{n}(-1)^i\Bb{i}{\sigma}=(-1)^{n_{M}(\sigma)-1}\Bb{n_{M}(\sigma)-1}{\sigma}.\end{equation} 

It is established in \cite{Bjoe} that for a matroid $M$ the dimension of $H_{i}(M;\field{k})$ is in fact independent of $\field{k}$. Thus \emph{for matroids, the ($\nn$- or $\nnn$-graded) Betti numbers are not only unique, but independent of the choice of field}. We shall therefore omit referring to or specifying a particular field $\field{k}$ throughout. 

\begin{example}[Continuation of Ex.~\ref{RunEx}]\label{RunEx2}
Since $M(H)$ has set of circuits 
\begin{flalign*}\big\{&\{ 1, 2, 6, 7 \},\{ 5, 6 \},\{ 2, 3, 6, 7 \},\{ 1, 2, 3, 5 \},\{ 1, 3, 7\},\{ 1, 4, 7 \},\{ 1, 2, 3, 6 \},\\&\{ 2, 4, 6 \},\{ 2, 3, 5, 7 \},\{ 3, 4, 7 \},\{ 1, 2, 5, 7 \},\{ 1, 3, 4 \},\{ 2, 4, 5 \}\big\}\end{flalign*}
its Stanley-Reisner ideal is \begin{flalign*}I_{M(H)}=\langle& x_1x_2x_6x_7,x_5x_6,x_2x_3x_6x_7,x_1x_2x_3x_5,x_1x_3x_7,x_1x_4x_7,x_1x_2x_3x_6,\\&x_2x_4x_6,x_2x_3x_5x_7,x_3x_4x_7,x_1x_2x_5x_7,x_1x_3x_4,x_2x_4x_5\rangle.\end{flalign*}
Using MAGMA (\cite{MAGMA}), we find the $\nn$-graded minimal free resolution of $I_{M(H)}$ to be \[\minCDarrowwidth6pt\begin{CD}0@<<<S(-2)\oplus S(-3)^6\oplus S(-4)^6@<<<S(-4)^5\oplus S(-5)^{28}@<<<S(-6)^{31}@<<<S(-7)^{10}@<<<0\end{CD}.\]
Similarly, we find the $\nn$-graded minimal free resolutions corresponding to the elongations of $M$ to be
\begin{flalign*}
&I_{\elo{M(H)}{1}}:\\ &\minCDarrowwidth18pt\begin{CD}0@<<<S(-4)^{2}\oplus S(-5)^{15}@<<<S(-6)^{29}@<<<S(-7)^{13}@<<<0\end{CD},\\
\\
&I_{\elo{M(H)}{2}}:\\ &\minCDarrowwidth18pt\begin{CD}0@<<<S(-6)^{7}@<<<S(-7)^{6}@<<<0\end{CD},\\
\\
&I_{\elo{M(H)}{3}}:\\ &\minCDarrowwidth18pt\begin{CD}0@<<<S(-7)@<<<0\end{CD}.\end{flalign*}
\end{example}

\section{Number of codewords of weight $j$}\label{Pol}
Let $C$ be a linear $[n,k]$-code over $\Ff{q}$, with a generator matrix $G=\begin{bmatrix}g_{i,j}\end{bmatrix}$ for $1\leq i\leq k$, $1\leq j\leq n$. Let $Q=q^r$ for some $r\in\N$. 
\begin{definition}
For $0\leq m\leq n$, let $A_{C,m}(Q)$ denote the number of words of weight $m$ in $C\otimes_{\Ff{q}}\Ff{Q}$.  
\end{definition}
Let $\mathbf{c}_j$ denote column $j$ of $G$. If $\mathbf{a}=(a_1,a_2,\dots,a_k)\in {\Ff{Q}}^{k}$, the codeword $\mathbf{a}\cdot G$ has weight $n$ if and only if \[{\mathbf{c}_{j}}^T\cdot \mathbf{a}\neq 0\] for all $1\leq j\leq n.$ In other words, if we let $S_j(Q)=\{\mathbf{x}\in \Ff{Q}^k:{\mathbf{c}_{j}}^T\cdot \mathbf{x}=0\}$, corresponding to column $j$, we have that $\mathbf{a}\cdot G$ has weight $n$ if and only if \begin{equation}\label{inex}\mathbf{a}\in \Ff{Q}^k\smallsetminus (S_1(Q)\cup S_2(Q)\cup\cdots\cup S_n(Q)).\end{equation}

\begin{definition}
For $U=\{u_1,u_2,\ldots,u_s\}\subseteq \{1,\ldots,n\}$, let \[S_U(Q)=S_{u_1}(Q)\cap S_{u_2}(Q)\cap\cdots\cap S_{u_s}(Q).\] 
\end{definition}
By the inclusion/exclusion-principle then, we see from (\ref{inex}) that \[A_{C,n}(Q)=Q^k-\sum_{|U|=1}|S_{U}(Q)|+\sum_{|U|=2}|S_{U}(Q)|+\cdots+(-1)^n\sum_{|U|=n}|S_U(Q)|.\] 
If $B_U=\begin{pmatrix}
\ \ \mathbf{c}_{u_1}^T \ \ \\ \ \ \ \mathbf{c}_{u_2}^T\ \ \ \\\vdots\\ \ \ \mathbf{c}_{u_s}^T \ \ \end{pmatrix}$, then $|S_U(Q)|=Q^{\dim(\ker B_U)}=Q^{k-\dim(\col B_U)}=Q^{k-r_{M(G)}(U)},$ which according to (\ref{dual}) is equal to $Q^{n_{M(H)}(E\smallsetminus U)}$. Since $Q^k=|S_{\emptyset}(Q)|\hide{\sum_{|U|=0}|S_U(Q)|}$, we conclude that \begin{equation}\label{an}A_{C,n}(Q)=\sum_{U\subseteq E}(-1)^{|U|}Q^{n_{M(H)}(E\smallsetminus U)}=(-1)^{n}\sum_{\gamma\subseteq E}(-1)^{|\gamma|}Q^{n_{M(H)}(\gamma)}.\end{equation} 
\hide{ 
In particular, if $H$ is a parity-check matrix for $C$ then \[A_{C,n}(0)=\sum_{n_{M(H)}(\gamma)=0}(-1)^{|\gamma|}=(-1)^{n+1}\chi(M(H)).\]} 

\begin{definition}\[a_{C,\sigma}(Q)=|\{w\in C\otimes_{\Ff{q}}\Ff{Q}: \Supp(w)=\sigma\}|.\]\end{definition}
\begin{lemma}\label{ak}
\[a_{C,\sigma}(Q)=(-1)^{|\sigma|}\sum_{\gamma\subseteq\sigma}(-1)^{|\gamma|}Q^{n_M(\gamma)}.\]
\end{lemma}
\begin{proof}
Let $C_{\sigma}(Q)$ denote the shortening of $C\otimes_{\Ff{q}}\Ff{Q}$ in $\{1\ldots n\}\smallsetminus \sigma$, and let $H_{|\sigma}$ be the restriction of $H$ to columns indexed by $\sigma$. Then $H_{|\sigma}$ is a parity-check matrix for $C_{\sigma}(Q)$. 

Clearly $a_{C,\sigma}(Q)=a_{C_{\sigma},\sigma}(Q),$ and since $M(H)_{|\sigma}\cong M(H_{|\sigma})$ it follows by an argument similar to the one leading to (\ref{an}) that \[a_{C,\sigma}(Q)=(-1)^{|\sigma|}\sum_{\gamma\subseteq\sigma}(-1)^{|\gamma|}Q^{n_{M(H)_{|\sigma}}(\gamma)}.\] The result follows, since $n_{M(H)_{|\sigma}}(\gamma)=n_{M(H)}(\gamma)$ for all $\gamma\subseteq \sigma$.
\end{proof}

\begin{proposition}\label{forkoder}For $1\leq m\leq n$, 
\[A_{C,m}(Q)=(-1)^m\sum_{|\sigma|=m}\sum_{\gamma\subseteq\sigma}(-1)^{|\gamma|}Q^{n_{M(H)}(\gamma)}\] 
\end{proposition}
\begin{proof}
This is clear from Lemma \ref{ak}, since $A_{C,m}(Q)=\sum_{|\sigma|=m}a_{C,\sigma}(Q)$.  
\end{proof}
In the following sections we shall see what comes from generalizing the \emph{weight polynomials} $A_{C,m}(Q)$ to matroids.

\section{Generalized weight polynomials and a generalized enumerator}\label{SecGWP}
Looking back at Proposition \ref{forkoder}, it is clear that the polynomial $A_{C}(Q)$ appearing there may equally well be defined for matroids in general -- not only for those derived from a linear code. 

For the remainder of this section, let $M$ be a matroid on $E=\{1,\ldots,n\}$.
\subsection{GWP and the enumerator}
\begin{definition}[GWP]\label{GWP}
We define the polynomial $P_{M,j}(Z)$ by letting $P_{M,0}(Z)=1$ and \[P_{M,j}(Z)=(-1)^j\sum_{|\sigma|=j}\sum_{\gamma\subseteq\sigma}(-1)^{|\gamma|}Z^{n_M(\gamma)} \text{ for }1\leq j\leq n.\] We shall refer to $P_{M,j}$ as the $j^{th}$ \emph{generalized weight polynomial}, or just \emph{GWP}, of $M$. 
\end{definition}
In light of Proposition \ref{forkoder}, we see that $A_{C,j}(q^{r})=P_{M(H),j}(q^{r})$ for any linear $\Ff{q}$-code $C$ with parity check matrix $H$. 

Comparing Definition \ref{GWP} with the definition of $d_i(M)$, it is immediately clear that the generalized weight polynomials together determine the higher weights:
\begin{proposition}\label{gwpbestemmerdi}
\[d_i(M)=\min\{j:\deg P_{M,j}=i\}.\]  
\end{proposition}

Also, we would like to point out that the $n^{th}$ generalized weight polynomial of $M$ is equal to the \emph{characteristic polynomial} (see \cite{Ku}) of $\dual{M}$.

Analogous to how $A_{C,j}(Q)$ is used to define the \emph{extended weight enumerator} $W_{C}(X,Y,Q)$ of a code $C$ (see \cite{JP}), we use the GWP to define the \emph{enumerator} of $M$:
\begin{definition}[Matroid enumerator]
The \emph{enumerator} $W_{M}$ of $M$ is \[W_{M}(X,Y,Z)=\sum_{i=0}^{n}P_{M,i}(Z)X^{n-i}Y^i.\] 
\end{definition}
\begin{example}\label{Vamos}
Let $\mathcal{V}^8$ be the matroid on $E=\{1,\ldots,8\}$ with bases \[\{\sigma\subseteq E: |\sigma|=4\}\smallsetminus \big\{\{1,2,3,4\},\{1,2,7,8\},\{3,4,5,6\},\{3,4,7,8\},\{5,6,7,8\}\big\}.\] This is the well-known \emph{V\'amos matroid}. It is non-representable; that is, it is not the vector matroid of any matrix (and thus does not come from any code). Using MAGMA, we find the enumerator of $\mathcal{V}^8$ to be \begin{flalign*}W_{\mathcal{V}^8}(X,Y,Z)=&X^8 + 5X^4Y^4Z - 5X^4Y^4 + 36X^3Y^5Z - 36X^3Y^5 + 28X^2Y^6Z^2\\& -138X^2Y^6Z + 110X^2Y^6 + 8XY^7Z^3 - 56XY^7Z^2 + 148XY^7Z\\& -100XY^7 + Y^8Z^4 - 8Y^8Z^3 + 28Y^8Z^2 - 51Y^8Z + 30Y^8.\end{flalign*}
\end{example}

Observe that if $C$ is a linear code with parity-check matrix $H$ and extended weight enumerator $W_{C}(X,Y,Q)$ (see e.g.~\cite{JP}), then \[W_{C}(X,Y,Q)=W_{M(H)}(X,Y,Q).\] 

\subsection{Equivalence to the Tutte polynomial}
It was shown in \cite{JP} that the extended weight enumerator of a linear code is equivalent to the Tutte polynomial of its associated matroid. We shall see that this is still true when it comes to matroids and their enumerators, in general. After a small leap (Proposition \ref{jan}), an analogous proof to the one found in \cite{JP} for linear codes went through.  
\begin{proposition}\label{jan}\[P_{M,i}(Z)=\sum_{j=n-i}^{n}(-1)^{i+j+n}\binom{j}{n-i}\sum_{|\gamma|=j}Z^{n_{M}(E\smallsetminus \gamma)}.\]
\end{proposition}
\begin{proof}
\begin{flalign*}P_{M,i}(Z)&=(-1)^i\sum_{|\sigma|=i}\sum_{\gamma\subseteq\sigma}(-1)^{|\gamma|}Z^{n_M(\gamma)}\\&=(-1)^i\sum_{|\sigma|=i}\sum_{E\smallsetminus \gamma\subseteq\sigma}(-1)^{|E\smallsetminus \gamma|}Z^{n_M(E\smallsetminus \gamma)}\\&=(-1)^i\sum_{|\sigma|=i}\sum_{E\smallsetminus \sigma\subseteq \gamma}(-1)^{|E\smallsetminus \gamma|}Z^{n_M(E\smallsetminus \gamma)}\\&=(-1)^i\sum_{|\gamma|\geq n-i}\sum_{\left\{\substack{E\smallsetminus\sigma:\\ E\smallsetminus\sigma\subseteq \gamma,\\ |\sigma|=i}\right\}}(-1)^{|E\smallsetminus \gamma|}Z^{n_M(E\smallsetminus \gamma)}\\&=(-1)^i\sum_{|\gamma|\geq n-i}\binom{j}{n-i}(-1)^{|E\smallsetminus \gamma|}Z^{n_M(E\smallsetminus \gamma)}\\&=\sum_{j=n-i}^{n}\sum_{|\gamma|=j}\binom{j}{n-i}(-1)^{i+j+n}Z^{n_M(E\smallsetminus \gamma)}.\end{flalign*} 
\end{proof}
Proposition \ref{jan} above is what enables us to use basically the same technique as that employed in \cite{JP} for the proofs of Theorems \ref{Tutte1} and \ref{Tutte2}.
\begin{lemma}\label{r1}
\[W_{M}(X,Y,Z)=\sum_{j=0}^{n}\sum_{|\gamma|=j}Z^{n_{M}(E\smallsetminus \gamma)}(X-Y)^{j}Y^{n-j}.\] 
\end{lemma}
\begin{proof} 
\begin{flalign*}W_{M}(X,Y,Z)&=\sum_{i=0}^{n}\sum_{j=n-i}^{n}(-1)^{i+j+n}\binom{j}{n-i}\sum_{|\gamma|=j}Z^{n_M(E\smallsetminus \gamma)}X^{n-i}Y^{i}\\&=\sum_{j=0}^{n}\sum_{i=n-j}^{n}(-1)^{j-n+i}\binom{j}{j-n+i}\sum_{|\gamma|=j}Z^{n_M(E\smallsetminus \gamma)}X^{n-i}Y^{i}\\&=\sum_{j=0}^{n}\sum_{k=0}^{j}(-1)^{k}\binom{j}{k}\sum_{|\gamma|=j}Z^{n_M(E\smallsetminus \gamma)}X^{j-k}Y^{n-j+k}\\&=\sum_{j=0}^{n}\sum_{|\gamma|=j}Z^{n_M(E\smallsetminus \gamma)}\left(\sum_{k=0}^{j}(-1)^{k}\binom{j}{k}X^{j-k}Y^{k}\right)Y^{n-j}\\&=\sum_{j=0}^{n}\sum_{|\gamma|=j}Z^{n_M(E\smallsetminus \gamma)}(X-Y)^{j}Y^{n-j}.\end{flalign*}\end{proof} 

We shall also need a slight reformulation of the Tutte polynomial. 
\begin{lemma}\label{lemTutte}
\[t_{M}(X,Y)=\sum_{j=0}^{n}\sum_{|\gamma|=j}(X-1)^{n_{M}(E\smallsetminus \gamma)-(n-r(M)-j)}(Y-1)^{n_{M}(E\smallsetminus \gamma)}.\] 
\end{lemma}
\begin{proof}
Follows by rewriting the $t_{M}(X,Y)$ from Definition \ref{Tuttedef} as \[t_{M}(X,Y)=\sum_{\gamma\subseteq E}(X-1)^{\dual{n}(\gamma)}(Y-1)^{n(E\smallsetminus \gamma)}\] and noting that $\dual{n}(\gamma)=n(E\smallsetminus \gamma)-(\dual{r}(E)-|\gamma|)$.  

\end{proof}
Using Lemmas \ref{r1} and \ref{lemTutte} it is now routine to verify the following two identities:
\begin{theorem}\label{Tutte1}
\[W_{M}(X,Y,Z)=(X-Y)^{n-r(M)}Y^{r(M)}t_{M}\left(\frac{X}{Y},\frac{X+(Z-1)Y}{X-Y}\right).\] 
\end{theorem}

\begin{theorem}\label{Tutte2}
\[t_{M}(X,Y)=(X-1)^{-(n-r(M))}X^{n}W_{M}(1,X^{-1},(X-1)(Y-1)).\]
\end{theorem}

\begin{example}[Continuation of Ex.~\ref{Vamos}]
Having already found the weight enumerator of $\mathcal{V}^8$, we infer from Theorem \ref{Tutte2} that \[t_{\mathcal{V}^8}(X,Y)=X^4 + 4X^3 + 10X^2 + 5XY + 15X + Y^4 + 4Y^3 + 10Y^2 + 15Y.\] 
\end{example}

\section{The GWP is determined by $\nn$-graded Betti numbers}\label{BettiSect}
As before, let $M$ be a matroid on $E=\{1,\ldots,n\}$. Recall from Section \ref{Secprelim} that the $\nn$- and $\nnn$-graded Betti numbers corresponding to the Stanley-Reisner ideal $I_{M}$ are independent of the choice of the underlying field $\field{k}$. The only thing of importance, and thus our only assumption, is that the $\nn$-graded (or $\nnn$-graded) minimal free resolution of $I_{M}$ is constructed with respect to the same field as the reduced chain complex over $M$. We may therefore omit specifying a field. Recall also that $\elo{M}{l}$ denotes the elongation of $M$ to rank $r(M)+l$.

Throughout the rest of this article we shall employ the convention that $\Bnum{i}{j}{\elo{M}{l}}=0$ whenever $l\notin[0,n-r(M]$.
\begin{theorem}[Main result]\label{main}For each $1\leq j\leq n$ the coefficient of $Z^l$ in $P_{M,j}$ is equal to \[\sum_{i=0}^{n}(-1)^{i}\Big(\Bnum{i}{j}{\elo{M}{l-1}}-\Bnum{i}{j}{\elo{M}{l}}\Big).\] 
\end{theorem}
\begin{proof}
Let $s_{\sigma,l}$ denote the coefficient of $Z^l$ in $P_{M_{|\sigma},|\sigma|}$. Since \[P_{M,j}(Z)=\sum_{|\sigma|=j}P_{M_{|\sigma},|\sigma|}(Z),\] the coefficient of $Z^l$ in $P_{M,j}(Z)$ is $\sum_{|\sigma|=j}s_{\sigma,l}$. On the other hand, we have \[s_{\sigma,l}=(-1)^{|\sigma|}\sum_{\substack{\gamma\subseteq \sigma\\n_{M}(\gamma)=l}}(-1)^{|\gamma|}=(-1)^{|\sigma|}\Big[\sum_{\substack{\gamma\subseteq \sigma\\n_{\elo{M}{l}}(\gamma)=0}}(-1)^{|\gamma|}-\sum_{\substack{\gamma\subseteq \sigma\\n_{\elo{M}{l-1}}(\gamma)=0}}(-1)^{|\gamma|}\Big].\] 
Applying Theorems \ref{BjoLemma} and \ref{Hochster}, in combination with (\ref{switch}), we see that \begin{flalign*}
(-1)^{|\sigma|}\sum_{\gamma\subseteq \sigma, n_{\elo{M}{l}}(\gamma)=0}(-1)^{|\gamma|}&=(-1)^{n_{\elo{M}{l}}(\sigma)}\dim H_{r_{\elo{M}{l}}(\sigma)}(M_{(l)|\sigma})\\
&=(-1)^{n_{\elo{M}{l}}(\sigma)}\Bb{n_{\elo{M}{l}}(\sigma)-1}{\sigma}(I_{M_{(l)|\sigma}}),
\end{flalign*}
which is equal to $(-1)^{n_{\elo{M}{l}}(\sigma)}\Bnum{n_{\elo{M}{l}}(\sigma)-1}{\sigma}{\elo{M}{l}}$ -- since, in general, $\Bb{i}{\sigma}(\Delta)=\Bb{i}{\sigma}(\Delta_{|\sigma})$. 

Thus \begin{flalign*}
s_{\sigma,l}&=(-1)^{n_{\elo{M}{l}}(\sigma)}\Bnum{n_{\elo{M}{l}}(\sigma)-1}{\sigma}{\elo{M}{l}}-(-1)^{n_{\elo{M}{l-1}}(\sigma)}\Bnum{n_{\elo{M}{l-1}}(\sigma)-1}{\sigma}{\elo{M}{l-1}}\\
&=(-1)^{n_{\elo{M}{l-1}}(\sigma)-1}\Bnum{n_{\elo{M}{l-1}}(\sigma)-1}{\sigma}{\elo{M}{l-1}}-(-1)^{n_{\elo{M}{l}}(\sigma)-1}\Bnum{n_{\elo{M}{l}}(\sigma)-1}{\sigma}{\elo{M}{l}},\end{flalign*}
which by (\ref{eple}) is equal to \[\sum_{i=0}^{n}(-1)^{i}\Bnum{i}{\sigma}{\elo{M}{l-1}}-\sum_{i=0}^{n}(-1)^{i}\Bnum{i}{\sigma}{\elo{M}{l}}.\] Consequently, the coefficient of $Z^l$ in $P_{M,j}(Z)$ is \begin{flalign*}\sum_{|\sigma|=j}\left(\sum_{i=0}^{n}(-1)^{i}\Big(\Bnum{i}{\sigma}{\elo{M}{l-1}}-\Bnum{i}{\sigma}{\elo{M}{l}}\Big)\right)&=\sum_{i=0}^{n}(-1)^{i}\left(\sum_{|\sigma|=j}\Bnum{i}{\sigma}{\elo{M}{l-1}}-\sum_{|\sigma|=j}\Bnum{i}{\sigma}{\elo{M}{l}}\right)\\&=\sum_{i=0}^{n}(-1)^{i}\Big(\Bnum{i}{j}{\elo{M}{l-1}}-\Bnum{i}{j}{\elo{M}{l}}\Big).\end{flalign*} 
\end{proof}
\begin{example}[Continuation of Ex.~\ref{RunEx2}]\label{RunEx3}
Let us calculate $P_{M(H),5}(Z)$ using Theorem \ref{main}. Having already found the $\nn$-graded Betti numbers of $M(H)$ and its elongations, we easily calculate 
\begin{flalign*} P_{M(H),5}(Z)=&\Bnum{0}{5}{\elo{M(H)}{1}}Z^{2}\\
&+\Big(\big(-\Bnum{1}{5}{M(H)}\big)-\big(\Bnum{0}{5}{\elo{M(H)}{1}}\big)\Big)Z\\
&-\big(-\Bnum{1}{5}{M(H)}\big)\\
=&15Z^{2}+\big((-28)-15\big)Z-(-28).
\end{flalign*} Continuing like this, we find the complete set of weight polynomials:
\begin{flalign*}&P_{M(H),0}(Z)=1\\&P_{M(H),1}(Z)=0\\&P_{M(H),2}(Z)=Z - 1&\\&P_{M(H),3}(Z)=6Z - 6\\&P_{M(H)4}(Z)=2Z^2 - Z - 1\\&P_{M(H),5}(Z)=15Z^2 - 43Z + 28\\&P_{M(H),6}(Z)=7Z^3 - 36Z^2 + 60Z - 31\\&P_{M(H),7}(Z)=Z^4 - 7Z^3 + 19Z^2 - 23Z + 10.\end{flalign*} 
\end{example}

\begin{corollary}\label{Wpol}Let $C$ be a linear code over $\Ff{q}$ of length $n$, with parity check matrix $H$. For $1\leq m\leq n$ and $Q$ a power of $q$, we have
\[A_{C,m}(Q)=\sum_{l=0}^{n}\left(\sum_{i=0}^{n}(-1)^{i}\Big(\Bnum{i}{m}{\elo{M(H)}{l-1}}-\Bnum{i}{m}{\elo{M(H)}{l}}\Big)\right)Q^{l}.\] 
\end{corollary}
\begin{proof}
This is immediate from Theorem \ref{main}, since $A_{C,m}(Q)=P_{M(H),m}(Q)$ by Proposition \ref{forkoder}. 
\end{proof}
In light of Corollary \ref{Wpol}, the polynomials found in Example \ref{RunEx3}, when evaluated at $q^{r}$, determine the number of codewords of a given weight in $C\otimes_{\Ff{q}}\Ff{q^{r}}$.

Occasionally, the result of Corollary \ref{Wpol} can greatly simplify the task of calculating weight polynomials $A_{C,m}(Q)$ for a linear code $C$. This is for instance the case with MDS-codes:
\begin{example}
Let $C$ be an MDS $[n,k]$-code over $\field{F}_{q}$, with parity check matrix $H$. It is well known that $M(H)$ is the uniform matroid $U(s,n)$, where $s=n-k$; which of course implies that \[\elo{M(H)}{l}=U(s+l,n).\]
From e.g.~\cite[Example 3]{JV}, we see that \[\Bnum{i}{j}{\elo{M(H)}{l}}=\begin{cases}
\binom{j-1}{s+l}\binom{n}{j}, i=j-l-s-1.\\
0, \text{otherwise}.\\
\end{cases}\]
We conclude from Corollary \ref{Wpol} that for $1\leq m\leq n$, and $Q=q^{r}$, we have \begin{flalign*}A_{C,m}(Q)=&\sum_{l=1}^{n}(-1)^{m+l+s}\binom{n}{m}\left(\binom{m-1}{s+l-1}+\binom{m-1}{s+l}\right)Q^{l}+(-1)^{m+s}\binom{n}{m}\binom{m-1}{s}\\=&(-1)^{m+s}\binom{n}{m}\left(\sum_{l=1}^{n}(-1)^{l}\binom{m}{s+l}Q^{l}+\binom{m-1}{s}\right).\end{flalign*}
\end{example}

\subsection{Further results}
The generalized weight polynomials of $\elo{M}{k-1}$ determine those of $\elo{M}{k}$ for all $1\leq k\leq n-r(M)$.
\begin{proposition}\label{jan2}
Let $k\geq 1$. If \[P_{\elo{M}{k-1},j}(Z)=a_{n}Z^{n}+a_{n-1}Z^{n-1}+\cdots+a_{1}Z+a_{o},\] then \[P_{\elo{M}{k},j}(Z)=a_{n}Z^{n-1}+a_{n-1}Z^{n-2}+\cdots+a_{2}Z+(a_{1}+a_{o}).\] 
\end{proposition}
\begin{proof}
Let $s_{\sigma,l}^{(k)}$ denote the coefficient of $Z^{l}$ in $P_{\elo{M}{k}|_{\sigma},|\sigma|}$. As noted in the proof of Theorem \ref{main}, the coefficient of $Z^{l}$ in $P_{\elo{M}{k},j}$ is $\sum_{|\sigma|=j}s_{\sigma,l}^{(k)}$, and \[s_{\sigma,l}^{(k)}=(-1)^{|\sigma|}\sum_{\substack{\gamma\subseteq \sigma\\n_{\elo{M}{k}}(\gamma)=l}}(-1)^{|\gamma|}.\] 

Assume first that $l\geq1$. By Proposition \ref{ni1} we have, \begin{flalign*}s_{\sigma,l}^{(k)}&=(-1)^{|\sigma|}\sum_{\substack{\gamma\subseteq \sigma\\n_{\elo{M}{k}}(\gamma)=l}}(-1)^{|\gamma|}\\&=(-1)^{|\sigma|}\sum_{\substack{\gamma\subseteq \sigma\\n_{\elo{M}{k-1}}(\gamma)=l+1}}(-1)^{|\gamma|}\\&=s_{\sigma,l+1}^{(k-1)}.\end{flalign*} Finally, by Proposition \ref{ni1} again, we see that \begin{flalign*}s_{\sigma,0}^{(k)}&=(-1)^{|\sigma|}\sum_{\substack{\gamma\subseteq \sigma\\n(\gamma)\leq k}}(-1)^{|\gamma|}\\&=(-1)^{|\sigma|}\sum_{\substack{\gamma\subseteq \sigma\\n(\gamma)= k}}(-1)^{|\gamma|}+(-1)^{|\sigma|}\sum_{\substack{\gamma\subseteq \sigma\\n(\gamma)\leq k-1}}(-1)^{|\gamma|}\\&=(-1)^{|\sigma|}\sum_{\substack{\gamma\subseteq \sigma\\n_{\elo{M}{k-1}}(\gamma)=1}}(-1)^{|\gamma|}+(-1)^{|\sigma|}\sum_{\substack{\gamma\subseteq \sigma\\n_{\elo{M}{k}}(\gamma)=0}}(-1)^{|\gamma|}\\&=s_{\sigma,1}^{(k-1)}+s_{\sigma,0}^{(k-1)},\end{flalign*} and this concludes our proof.	
\end{proof}

Combining Propositions \ref{jan2} and \ref{gwpbestemmerdi}, we see that
\begin{corollary}\label{vekterMi}
\[d_{i}(\elo{M}{l+1})=d_{i+1}(\elo{M}{l}).\]
\end{corollary}

\begin{example}[The simplex code $\mathcal{S}_{2}(3)$]
Let $\mathcal{S}_{2}(3)$ be the simplex code of dimension $3$ over $\Ff{2}$. This code has length $n=7$. Let $H$ be a parity-check matrix of $\mathcal{S}_{2}(3)$. 

The higher weights of $\mathcal{S}_{2}(3)$ are $(d_{1},d_2,d_3)=(4,6,7)$, from which it follows by way of \cite[Theorem 2]{JV2} that the non-zero Betti numbers of $I_{M(H)}$ are \[(\beta_{0,4},\beta_{1,6},\beta_{2,7})=(7,14,8).\] By Proposition \ref{vekterMi}, the higher weights of $\elo{M}{1}$ are $(d_1,d_2)=(6,7)$, which implies that $\elo{M(H)}{1}$ must be the uniform matroid $U(5,7)$. From \cite[Example 3]{JV} then, we see that the only non-zero Betti numbers of $I_{\elo{M(H)}{1}}$ are $\Bnum{0}{6}{\elo{M(H)}{1}}=7$ and $\Bnum{1}{7}{\elo{M(H)}{1}}=6$. As always, the $(n-r(M(H))-1)^{th}$ elongation $\elo{M(H)}{2}$ has $\{1,\ldots,7\}$ as its only circuit, such that the only non-zero Betti number associated with $I_{\elo{M(H)}{2}}$ is $\Bnum{0}{7}{\elo{M(H)}{2}}=1$. 

Having found all $\nn$-graded Betti numbers from all elongations, we easily calculate the weight polynomials using Corollary \ref{Wpol}:	
\begin{flalign*} 
&A_{\mathcal{S}_{2}(3),0}(Q)=1\\
&A_{\mathcal{S}_{2}(3),1}(Q)=0\\
&A_{\mathcal{S}_{2}(3),2}(Q)=0\\
&A_{\mathcal{S}_{2}(3),3}(Q)=0\\
&A_{\mathcal{S}_{2}(3),4}(Q)=7Q-7\\
&A_{\mathcal{S}_{2}(3),5}(Q)=0\\
&A_{\mathcal{S}_{2}(3),6}(Q)=7Q^2-21Q+14\\
&A_{\mathcal{S}_{2}(3),7}(Q)=Q^3-7Q^2+14Q-8\\
\end{flalign*}
\end{example}

\section{Concerning the converse}\label{Counter}
Having seen that the Betti numbers associated with the elongations $\elo{M}{i},0\leq i\leq n-r(M)$, determine the polynomials $P_{M,j}(Z),1\leq j\leq n$, it is natural to ask whether the opposite is true. The answer to this is negative, as the following counterexample shows:
\begin{example}[Continuation of Ex.~\ref{RunEx3}]\label{RunEx4}
Let $N$ be the matroid on $\{1,\ldots,7\}$ with bases 
\begin{flalign*}B(N)=\big\{&\{ 1, 4, 7 \},\{ 1, 3, 6 \},\{ 1, 3, 5 \},\{ 1, 3, 4 \},\{ 2, 3, 6 \},\{ 3, 4, 7 \},\{ 1, 2, 5 \},\{ 1, 5, 7 \},\\ &\{ 3, 6, 7 \},\{ 2, 4, 7 \},\{ 3, 5, 6 \},\{ 2, 3, 4 \}, \{ 1, 2, 3 \},\{ 1, 2, 7 \},\{ 1, 5, 6 \},\{ 3, 4, 5 \},\\ &\{ 1, 6, 7 \},\{ 1, 4, 5 \},\{ 2, 3, 7 \},\{ 2, 5, 6 \},\{ 2, 4, 5 \},\{ 3, 5, 7 \},\{ 2, 6, 7 \},\{ 2, 5, 7 \}\big\}.\end{flalign*} The Stanley-Reisner ideal of $N$ has minimal free resolution
\[\minCDarrowwidth5pt\begin{CD}
0@<<<S(-2)\oplus S(-3)^6\oplus S(-4)^5@<<<S(-4)^4\oplus S(-5)^{28}@<<<S(-6)^{31}@<<<S(-7)^{10}@<<<0\minCDarrowwidth5pt\end{CD}.\] Comparing to the minimal free resolution of $I_{M(H)}$, we see that the Betti numbers are not the same. However, it is easy to see, using Proposition \ref{main}, that $N$ has the same generalized weight polynomials as $M(H).$ 

Note that this is the ``smallest'' counterexample, in that there are no counterexamples for $n<7$. 
\end{example}

Moreover, knowing the Betti numbers of $M$ is in itself \emph{not} enough to calculate $P_{M,j}$ -- in general we need the Betti numbers derived from the elongations $\elo{M}{i}$ as well:
\begin{example}
The matroids $M$ and $N$ on $\{1,\ldots, 8\}$ with bases 
\begin{flalign*}B(M)=\big\{
        &\{ 1, 3, 4, 6, 7 \},
        \{ 1, 2, 3, 6, 8 \},
        \{ 1, 2, 3, 4, 8 \},
        \{ 1, 2, 3, 5, 8 \},
        \{ 1, 2, 5, 6, 8 \},
        \{ 1, 2, 3, 4, 7 \},\\
        &\{ 1, 2, 3, 5, 7 \},
        \{ 1, 2, 5, 6, 7 \},
        \{ 1, 3, 4, 5, 7 \},
        \{ 1, 3, 4, 6, 8 \},
        \{ 1, 2, 4, 6, 8 \},
        \{ 1, 2, 4, 6, 7 \},\\
        &\{ 1, 3, 4, 5, 8 \},
        \{ 1, 2, 4, 5, 7 \},
        \{ 1, 4, 5, 6, 7 \},
        \{ 1, 2, 3, 6, 7 \},
        \{ 1, 3, 5, 6, 7 \},
        \{ 1, 4, 5, 6, 8 \},\\
        &\{ 1, 3, 5, 6, 8 \},
        \{ 1, 2, 4, 5, 8 \}\big\}
\end{flalign*} and 
\begin{flalign*}B(N)=\big\{
&\{ 1, 3, 4, 6, 7 \},
        \{ 1, 2, 3, 4, 8 \},
        \{ 1, 2, 3, 5, 8 \},
        \{ 1, 2, 5, 6, 8 \},
        \{ 1, 2, 3, 4, 7 \},
        \{ 1, 2, 3, 5, 7 \},\\
        &\{ 1, 2, 5, 6, 7 \},
        \{ 1, 3, 4, 5, 7 \},
        \{ 1, 3, 4, 6, 8 \},
        \{ 1, 2, 4, 6, 8 \},
        \{ 1, 2, 4, 6, 7 \},
        \{ 1, 3, 4, 5, 8 \},\\
        &\{ 1, 2, 4, 5, 7 \},
        \{ 1, 3, 4, 5, 6 \},
        \{ 1, 2, 4, 5, 6 \},
        \{ 1, 3, 5, 6, 7 \},
        \{ 1, 2, 3, 5, 6 \},
        \{ 1, 2, 3, 4, 6 \},\\
        &\{ 1, 3, 5, 6, 8 \},
        \{ 1, 2, 4, 5, 8 \}
        \big\},
\end{flalign*} respectively, both have \[\minCDarrowwidth11pt\begin{CD}
0@<<<S(-2)\oplus S(-4)^5@<<<S(-5)^4\oplus S(-6)^5@<<<S(-7)^4@<<<0\minCDarrowwidth11pt\end{CD}\] as the minimal free resolution of their associated Stanley-Reisner ideal, while \[P_{M,4}(Z)=Z^2 - 5Z + 4\neq2Z^2 - 6Z + 4 =P_{N,4}(Z).\] 
Again this is the ``smallest'' counterexample.
\end{example}

It is however possible for two non-isomorphic matroids to have identical $\nn$-graded Betti numbers in all elongation levels (the smallest example of which is given by a pair of rank $3$ on $\{1,\ldots,6\}$).

\end{document}